\documentclass[12pt]{article}

\usepackage{amsthm,amsmath,amssymb}
\usepackage[noadjust,sort]{cite}

\usepackage{levelsubs}

\date{}

\normalsize

\newtheorem{thm}{Theorem}
\newtheorem{cor}[thm]{Corollary}
\newtheorem{lemma}[thm]{Lemma}

\def\dfrac#1#2{\lower0.15ex\hbox{\large$\frac{#1}{#2}$}}
\def\({\bigl(}
\def\){\bigr)}

\def\sumppd{\mathop{\sum\nolimits'}\limits}
\def\sumpp{\sum'}
\def\sumjk{\sum_{1\le j<k\le n}}
\def\dmax{d_{\mathrm{max}}}
\def\xmax{x_{\mathrm{max}}}

\def\c{{\mathord{\ast}}}       

\def\II{\boldsymbol{I}}
\def\XX{\boldsymbol{X}}
\def\VV{\boldsymbol{V}}
\def\UU{\boldsymbol{U}}
\def\YY{\boldsymbol{Y}}
\def\Zero{\boldsymbol{0}}

\def\barXX{{\mkern 5mu\overline{\mkern-5mu\XX\mkern-3mu}\mkern 3mu}}

\def\A{{\cal A}}

\def\R{{\cal R}}
\def\S{{\cal S}}
\def\T{{\cal T}}

\let\eps=\varepsilon

\def\Deltait{\mathit{\Delta}}
\def\Thetait{\mathit{\Theta}}

\def\avec{\boldsymbol{a}}
\def\dvec{\boldsymbol{d}}
\def\xvec{\boldsymbol{x}}
\def\yvec{\boldsymbol{y}}
\def\zvec{\boldsymbol{z}}
\def\thetavec{\boldsymbol{\theta}}

\def\mw#1{\hat{#1}}
\def\mwA{\mw{A}}
\def\mwB{\mw{B}}
\def\mwC{\mw{C}}
\def\mwD{\mw{D}}
\def\mwE{\mw{E}}
\def\mwF{\mw{F}}
\def\mwG{\mw{G}}
\def\mwH{\mw{H}}
\def\mwI{\mw{I}}
\def\mwJ{\mw{J}}
\def\mwZ{\mw{Z}}
\def\mwn{N}
\def\mwa{\mw{a}}

\def\OO{\widetilde{O}}

\def\miss{\operatorname{miss}}
\def\hit{\operatorname{hit}}
\def\num{\operatorname{num}}

\def\Im{\operatorname{Im}}

\def\Prob{\operatorname{Prob}}
\def\Reals{{\mathbb{R}}}

\def\abs#1{\lvert#1\rvert} \let\card=\abs

\def\suchthat{\mathrel{:}}
\def\Suchthat{\mathrel{:}}
\def\nicebreak{\vskip 0pt plus 50pt\penalty-300\vskip 0pt plus -50pt }

\begin{document}

\title {Subgraphs of dense random graphs\\ with specified degrees}

\author{
Brendan~D.~McKay\vrule width0pt height2ex\thanks
 {Research supported by the Australian Research Council.}\\
\small School of Computer Science\\[-0.8ex]
\small Australian National University\\[-0.8ex]
\small Canberra ACT 0200, Australia\\[-0.3ex]
\small\texttt{bdm@cs.anu.edu.au}
}

\maketitle

\begin{abstract}
Let $\dvec=(d_1,d_2,\ldots, d_n)$ be a vector of non-negative
integers with even sum.
We prove some basic facts about the structure of a 
random graph with degree sequence~$\dvec$, including
the probability of a given subgraph or induced subgraph.

Although there are many results of this kind, they are
restricted to the sparse case with only a few exceptions.
Our focus is instead on the case where the average degree
is approximately a constant fraction of~$n$.

Our approach is the multidimensional saddle-point method.
This extends the enumerative work of McKay and Wormald
(1990) and is analogous to the theory developed for
bipartite graphs by Greenhill and McKay (2009).
\end{abstract}

\section{Introduction}\label{s:intro}

Let $\dvec=(d_1,d_2,\ldots,d_n)$ be a vector of non-negative
integers with even sum.
Let $\XX=(x_{jk})$ be a symmetric
$n\times n$ matrix over $\{0,1\}$ with zero diagonal.
Define
$G(\dvec,\XX)$ to be the number of $n\times n$ symmetric matrices
$\A=(a_{jk})$ over $\{ 0,1\}$ with zero diagonal, such that\\
(i) row $j$ sums to $d_j$, for $1\le j\le n$;\\
(ii) $a_{jk}=0$ whenever $x_{jk}=1$, for $1\le j,k\le n$.

Equivalently,
$G(\dvec,\XX)$ is the number of labelled simple graphs with
$n$ vertices of degree $d_1,d_2,\ldots, d_n$, having no edges
in common with the simple graph~$\XX$.
The special case where $\XX$ is the zero matrix $\Zero$ will also
be denoted $G(\dvec)$.
Define $\xvec=(x_1,x_2,\ldots,x_n)$, where $x_j$ is the sum of
the $j$th row of~$\XX$.

One motive for interest in $G(\dvec,\XX)$ is that the ratio
$G(\dvec,\XX)/G(\dvec)$ is the probability that a random simple
graph with degree sequence $\dvec$ has no edge in common with $\XX$.
Similarly, $G(\dvec-\xvec,\XX)/G(\dvec)$ is the probability
that $\XX$ appears as a subgraph.
In these cases, and throughout the paper,
probability spaces have the uniform distribution.

Define the matrix $\barXX=(\bar x_{jk})$ over $\{ 0,1\}$ with
$\bar x_{jk}=1$ iff $j\ne k$ and $x_{jk}=0$.
For convenience we will adopt the convention that
$\sum_{jk\in\XX}$ means the sum over all $\{j,k\}$ such that
$x_{jk}=1$, and similarly $\sum_{jk\in\barXX}$ means the sum over
all $\{j,k\}$ such that $\bar x_{jk}=1$.
Note that the equal sets $\{j,k\}$ and $\{k,j\}$ do not appear as
separate terms in these sums.

Define the following key parameters.
\def\fixwid#1{\hbox to 5em{$\displaystyle#1$\hss}}
\begin{align*}
  E &= \fixwid{\dfrac12 \sum_{j=1}^n d_j}
          \text{\quad(the number of edges)} \\
  d &= \fixwid{\frac{2E}{n}} \text{\quad(the average degree)} \\
  \lambda &= \fixwid{\frac{d}{n-1}} \text{\quad(the density ignoring the diagonal)} \\
  A   &= \tfrac12 \lambda (1{-}\lambda) \\
  X   &= \fixwid{\dfrac12\, \sum_{j=1}^n x_j} \text{\quad(the number of
    edges of $\XX$)}\\
  \delta_j &= d_j-d+\lambda x_j\quad(1\le j\le n)
\end{align*}

Direct asymptotic estimation of $G(\dvec,\XX)$ for nonzero $\XX$
has been previously restricted to the sparse range.
For representative results with bounded or very slowly growing degrees,
see Bollob\'as and McKay~\cite{Bollobas} and
Wormald~\cite{WormaldSurvey}.
For somewhat higher degrees we have the following.
Let $\dmax=\max_j d_j$, $\xmax=\max_j x_j$ and
$\Deltait=\dmax(\dmax+\xmax)$.

\begin{thm}[\cite{Perth}]\label{perth}
Suppose $\dmax\ge 1$ and $\Deltait=o(E)$.
Then, as $n\to\infty$,
\begin{align*}
  G(\dvec,\XX) = \frac{(2E)!}{E!\,2^E\prod_{j=1}^n d_j!}
   & \exp\biggl( -\frac{\sum_{j=1}^n d_j(d_j{-}1)}{4E}
     - \frac{\(\sum_{j=1}^n d_j(d_j{-}1)\)^2}{16E^2} \\
    &{\kern30mm} - \frac{\sum_{jk\in\XX} d_jd_k}{2E}
        + O(\Deltait^2/E)
   \biggr).
\end{align*}
\end{thm}
The error term in Theorem~\ref{perth} is $o(1)$ only under the
stronger condition that $\Deltait^2=o(E)$,
which implies that the graphs are quite sparse.
The special case $G(\dvec)$ was determined by McKay
and Wormald~\cite{MWsparse} under the weaker condition
$\dmax^3=o(E)$.

The probability $G(\dvec,\XX)/G(\dvec)$ of being edge-disjoint
from $\XX$, and the probability
$G(\dvec-\xvec,\XX)/G(\dvec)$ of containing $\XX$
as a subgraph are easily deduced from Theorem~\ref{perth}.
They can also be found directly over a sometimes
wider range of $\dvec$ values.  Let $(a)_b$ denote the falling
factorial.
The following is a consequence
of Theorems 2.9 and 2.10 of McKay~\cite{McKay81}.

\begin{thm}[\cite{McKay81}]\label{mckay81}
If\/ $\Deltait+X = o(E)$ then, as $n\to\infty$,
\[
 \frac{G(\dvec-\xvec,\XX)}{G(\dvec)}
   = \frac{\prod_{j=1}^n (d_j)_{x_j}}{2^X(E)_X}
     \,\exp\( O(\Delta X/E) \).
\]
\end{thm}

In the case of dense matrices, the first asymptotically precise
enumeration result was that of McKay and Wormald~\cite{MWreg}, who
proved Theorem~\ref{bigtheorem} (below) in the case of
\hbox{$\XX=\Zero$} with a slightly weaker error term.
This has been extended by Barvinok and Hartigan~\cite{BH}.  They
identify a symmetric matrix $(\lambda_{jk})$ over $\{0,1\}$
introduced in~\cite{MWreg}
(and used in greater generality in Section~\ref{s:calculations}) as the
matrix that maximises a certain entropy function.  They then express
the asymptotic value of $G(\dvec)$ as an effectively computable
function of $(\lambda_{jk})$ provided the values $\{\lambda_{jk}\}_{j\ne k}$
are uniformly bounded away from~0 and~1.  This forces the average
degree to be $\Theta(n)$ but allows a much greater variation of degrees
than we allow.  They also show that $(\lambda_{jk})$ matches a typical
graph with degree sequence $\dvec$ in a sense that we will describe in
Section~\ref{s:subgraphs}.  This theme is explored in a somewhat
different way by Chatterjee, Diaconis and Sly~\cite{CPA}.

Despite the absence of precise enumerative results for densities between
$o(n^{-1/2})$ and $c/\log n$, Krivelevich, Sudakov, Vu and
Wormald~\cite{KSVW01} determined several almost-sure properties of
random regular graphs over various ranges of density using a combination
of switchings and analysis.
Other such properties were determined by Boldi and Vigna~\cite{Boldi},
and Cooper, Frieze, Reed and Riordan~\cite{CFR,CFRR}.

More recently, Krivelevich, Sudakov and Wormald~\cite{KSW09}
determined the probability of small induced subgraphs in random
regular graphs of degree $(n-1)/2$ under some conditions on the
order and degree sequence of the subgraph.

The corresponding problems for bipartite graphs and digraphs were studied
by Greenhill and McKay~\cite{GMX}; see that paper for a bibliography.
The proof method in~\cite{GMX} is quite similar to that here.

\medskip

We will also have need for the following additional parameters, for
$1\le j\le n$ and $\ell,m\ge 1$.
\begin{align*}
  B_j &= \sum_{k|jk\in\XX} \delta_k, &
  R &= \sum_{j=1}^n (d_j-d)^2, \\
  R_\ell &= \sum_{j=1}^n \delta_j^\ell, &
  X_\ell &= \sum_{j=1}^n x_j^\ell, \displaybreak[0]\\
  D &= \sum_{jk\in\XX} \delta_j\delta_k, &
  H &= \sum_{jk\in\XX} x_jx_k, \\
  L &= \sum_{jk\in\XX} (\delta_j-x_j)(\delta_k-x_k), &
  C_{\ell,m} &= \sum_{j=1}^n \delta_j^\ell x_j^m, \\
  K &= \sum_{jk\in\XX} (d_j-d)(d_k-d).
\end{align*}

To calibrate and motivate our main enumeration result, we first
develop a na\"\i ve estimate of $G(\dvec,\XX)$ by extending an
idea introduced in~\cite{MWreg}.
Generate a random graph by independently creating an edge
$jk$ with probability $\lambda$ for each $jk\in\barXX$.
Each graph with $E$ edges (none in common with $\XX$)
appears with probability $\lambda^E(1-\lambda)^{\binom n2-X-E}$.
Moreover, the event $E_j$ that vertex $j$ has degree $d_j$ has
probability 
$\binom{n-1-x_j}{d_j}\lambda^{d_j}(1-\lambda)^{n-1-x_j-d_j}$
for each~$j$.  If we (incorrectly) assume that the events
$E_1,\ldots,E_n$ are independent, we obtain a guess for
$G(\dvec,\XX)$ as follows:
\begin{equation}\label{naive}
  \widehat G(\dvec,\XX) = 
     (1-\lambda)^{-X}
     \(\lambda^\lambda (1-\lambda)^{1-\lambda}\)^{\binom n2}
     \prod_{j=1}^n \binom{n{-}1{-}x_j}{d_j}.
\end{equation}
In~\cite{MWreg} it was proved that
\[
    G(\dvec) =\sqrt 2\,\widehat G(\dvec,\Zero)
  	\exp\biggl( \,\frac14 - \frac{R^2}{16A^2n^4}
	   + o(1) \biggr).
\]
under certain conditions on $\dvec$.
Our main result extends this to nonzero~$\XX$.


\begin{thm}\label{bigtheorem}
Let $a,b>0$ be constants such that $a+b<\frac12$. 
Then there is a constant $\eps=\eps(a,b)>0$ such that the
following holds.
Suppose that $d_j{-}d,x_j=O(n^{1/2+\eps})$
uniformly for $1\le j\le n$,
that $X=O(n^{1+2\eps})$, and that
\[
     \min\{d,n-d-1\} \ge \frac {n}{3a\log n}.
\]
for sufficiently large $n$. Then, as $n\to\infty$,
\begin{align*}
  G(\dvec,\XX) &=\sqrt 2\,\widehat G(\dvec,\XX)
  	\exp\biggl( \,\frac14 - \frac{R^2}{16A^2n^4}
		+ \frac{\lambda X^2}{(1{-}\lambda)n^2}
		- \frac{D}{2An^2} + O(n^{-b}) \biggr).
\end{align*}
\end{thm}

\begin{proof}
The proof of this theorem is the main task of the paper.
Here we will summarize
the main phases and draw their conclusions together. 
The basic idea is to identify $G(\dvec,\XX)$ as a
coefficient in a multivariable generating function and to extract
that coefficient using the saddle-point method.
In Section~\ref{s:calculations},
we write $G(\dvec,\XX)=P(\dvec,\XX)I(\dvec,\XX)$,
where $P(\dvec,\XX)$ is a rational expression and
$I(\dvec,\XX)$
is an integral in $n$ complex dimensions.  Both depend on the
location of the saddle point, which is the solution of some nonlinear
equations.  Those equations are solved in Section~\ref{s:radii},
and this leads to the value of $P(\dvec,\XX)$ in~\eqref{Pvalue}.
In~Sections~\ref{s:integral}--\ref{s:diagonalize},
the integral $I(\dvec,\XX)$ is estimated in a superset of a
small region $\R$ enclosing the origin (defined in~\eqref{Rdef})
and an equivalent small region $\R'$ enclosing $(\pi,\ldots,\pi)$.
The result is given by Lemma~\ref{Fintegral}.
Finally, in Section~\ref{s:boxing},
we note that the integral restricted to the exterior
of $\R\cup\R'$ is negligible.  The present theorem thus follows
from~\eqref{start}, \eqref{Pvalue} and
Lemmas~\ref{Fintegral} and~\ref{boxing}.

The proof in various places requires $\eps$ to be
sufficiently small, but there are only a finite number of such
places so we can choose $\eps(a,b)$ to satisfy all of them
at once.   Note that the theorem remains true if $\eps(a,b)$
is decreased, since the conditions become stronger.
\end{proof}

\medskip

Throughout the paper, the asymptotic notation $O(f(n))$ refers
to the passage of~$n$ to~$\infty$.
We also use a modified notation $\OO(f(n))$,
which is to be taken as a shorthand for any expression of
the form $O(f(n)n^{c\eps})$ with $c$ a numerical constant
(perhaps a different constant at each occurrence).
Under the assumptions of Theorem~\ref{bigtheorem}, we have
$\lambda^{-1},(1{-}\lambda)^{-1}=O(\log n)$.
This implies, if $c_1,c_2,c_3,c_4$ are constants, that
$\lambda^{c_1}(1-\lambda)^{c_2}n^{c_3+c_4\eps}=\OO(n^{c_3})$.

\nicebreak
\section{Subgraph probabilities}
\label{s:subgraphs}

Define functions
$\miss(\dvec,\XX)$ and $\hit(\dvec,\XX)$ as follows. The
probability that a random simple graph with degrees $\dvec$ has
no edges in common with $\XX$ is
\[
    (1-\lambda)^X\miss(\dvec,\XX),
\]
and the probability that it includes $\XX$ as a subgraph is
\[
    \lambda^X\hit(\dvec,\XX).
\]
In this section, we apply Theorem~\ref{bigtheorem} to
estimate these probabilities.
To avoid unnecessary messiness regarding the value of $\eps$,
we will suppose $\eps(a,b)$ in Theorem~\ref{bigtheorem}
is chosen to be small enough to satisfy the finite number of
places in this section
where a statement is only true if $\eps$ is small enough.

\begin{thm}\label{subprobs}
Under the conditions of Theorem~\ref{bigtheorem}, we have
\begin{align*}
 \miss(\dvec,\XX) &= \exp\biggl(
    \frac{\lambda X}{(1{-}\lambda)n}
  + \frac{\lambda X_2}{2(1{-}\lambda)n}
  + \frac{\lambda(1{-}2\lambda)X_3}{6(1{-}\lambda)^2n^2} 
  + \frac{\lambda X^2}{(1{-}\lambda)n^2} 
  - \frac{D}{\lambda(1{-}\lambda)n^2} \\
  &{\kern 16mm}
  - \frac{C_{1,1}}{(1-\lambda)n}
  - \frac{(1{-}2\lambda)C_{1,2}}{2(1{-}\lambda)^2n^2}
  - \frac{C_{2,1}}{2(1{-}\lambda)^2n^2}		
		 + O(n^{-b}) \biggr) 
\intertext{and}
 \hit(\dvec,\XX) &= \exp\biggl(
   \frac{(1{-}\lambda)X}{\lambda n}
   - \frac{(1{+}\lambda)X_2}{2\lambda n}
   - \frac{(1{+}\lambda)(1{+}2\lambda)X_3}{6\lambda^2n^2}
   + \frac{(1{-}\lambda)X^2}{\lambda n^2}\\
   &{\kern 16mm}
   - \frac{L}{\lambda(1{-}\lambda)n^2}
   + \frac{C_{1,1}}{\lambda n}
   + \frac{(1{+}2\lambda)C_{1,2}}{2\lambda^2n^2}
   - \frac{C_{2,1}}{2\lambda^2n^2}
   	 + O(n^{-b}) \biggr).
\end{align*}
\end{thm}

\begin{proof}
Since $(1-\lambda)^X\miss(\dvec,\XX)=G(\dvec,\XX)/G(\dvec)$,
the first part can
be obtained from Theorem~\ref{bigtheorem}. The second part can be
found in similar fashion, or by noting that the probability of a
random graph avoiding $\XX$ is the probability of the
complement of the graph having $\XX$ as a subgraph.
\end{proof}

If $\XX$ is not too dense,
the probabilities in Theorem~\ref{subprobs}
asymptotically match those for an ordinary random graph with edge
probability~$\lambda$.  Sufficient conditions are that
$\miss(\dvec,\XX)=1+o(1)$ if
\[
     \lambda X_2 + X \max\nolimits_j\,\abs{d_j-d} = o\((1-\lambda)n\),
\]
and $\hit(\dvec,\XX)=1+o(1)$ if
\[
     (1-\lambda) X_2 + X \max\nolimits_j\,\abs{d_j-d} = o\(\lambda n\).
\]
Both these sufficient conditions hold, for example, if
$X=O(n^{1/2-2\eps})$, or if $d_j{-}d$ and $x_j$ are uniformly
$O(n^\eps)$ for $1\le j\le n$ and $X=O(n^{1-2\eps})$.

Since Theorem~\ref{subprobs} is rather complex, we give some special
cases to facilitate its application.  We also give the value of
$\num(\dvec,\XX)$, which is the exponential factor
in Theorem~\ref{bigtheorem}.

\begin{cor}\label{flat}
Suppose the conditions of Theorem~\ref{bigtheorem} hold, and in
addition assume that $d_1=\cdots=d_n=d$.  Then
\begin{align*}
 \num(\dvec,\XX) &= \exp\biggl(\frac{1}{4}
    + \frac{\lambda(X^2{-}H)}{(1{-}\lambda)n^2}
    + O(n^{-b}) \biggr),\\
 \miss(\dvec,\XX) &= \exp\biggl(
 \frac{\lambda X}{(1{-}\lambda)n}
  - \frac{\lambda X_2}{2(1{-}\lambda)n}
  - \frac{\lambda(2{-}\lambda)X_3}{6(1{-}\lambda)^2n^2} \\
  &{\kern 33mm}
  + \frac{\lambda X^2}{(1{-}\lambda)n^2}
  - \frac{\lambda H}{(1{-}\lambda)n^2}
  + O(n^{-b}) \biggr), \displaybreak[0] \\
 \hit(\dvec,\XX) &= \exp\biggl(
   \frac{(1{-}\lambda)X}{\lambda n}
   - \frac{(1{-}\lambda)X_2}{2\lambda n}
   - \frac{(1{-}\lambda^2)X_3}{6\lambda^2n^2}\\
   &{\kern 33mm}
   + \frac{(1{-}\lambda)X^2}{\lambda n^2}
   - \frac{(1{-}\lambda)H}{\lambda n^2}
   + O(n^{-b}) \biggr).
\end{align*}
\end{cor}

\begin{cor}\label{reg}
Suppose the conditions of Theorem~\ref{bigtheorem} hold, and in
addition assume that $x_1=\cdots=x_n=x$ (which implies
that $x=O(n^{2\eps})$).  Then
\begin{align*}
 \num(\dvec,\XX) &= \exp\biggl(\frac{1}{4}
  + \frac{\lambda x^2}{4(1{-}\lambda)}
  - \frac{K}{2An^2}
  - \frac{R^2}{16A^2n^4}
  + O(n^{-b}) \biggr), \\
 \miss(\dvec,\XX) &= \exp\biggl(
  - \frac{\lambda x(x{-}2)}{4(1{-}\lambda)}
  - \frac{xR}{2(1{-}\lambda)^2n^2}
  - \frac{K}{2An^2}
  + O(n^{-b}) \biggr), \displaybreak[0]\\
 \hit(\dvec,\XX) &= \exp \biggl(
  - \frac{(1{-}\lambda)x(x{-}2)}{4\lambda}
  - \frac{xR}{2\lambda^2n^2}
  - \frac{K}{2An^2}
  + O(n^{-b}) \biggr).
\end{align*}
\end{cor}

The two parts of Theorem~\ref{subprobs} have a common generalization.
Let $\YY$ be a supergraph of~$\XX$.
Then the probability that a random graph with degrees $\dvec$ has
intersection with~$\YY$ equal to $\XX$ is
\[
   \frac{G(\dvec-\xvec,\YY)}{G(\dvec)}.
\]
If the degrees of $\YY$ are $y_1,\ldots,y_n$, with
$y_j=O(n^{1/2+\eps})$ uniformly over~$j$, and $\sum_{j=1}^n
y_j=O(n^{1+2\eps})$, then this probability can be computed using
two applications of Theorem~\ref{bigtheorem}. The resulting general
formula is rather complex, so we will be content with presenting the
special case where~$\YY$ consists of a single clique and otherwise
isolated vertices. This is the important case of an induced
subgraph.

Suppose that for some $m$, we have $x_{m+1}=\cdots=x_n=0$.
Let $\XX^{[m]}$ be the subgraph of $\XX$ induced by vertices
$1,\ldots,m$ (so $\XX^{[m]}$ has the same edges
as~$\XX$).  For $k,\ell\ge 0$, define the quantity
\[
     \omega_{k,\ell}
       = \sum_{j=1}^m (d_j-d)^k(x_j-\lambda(m-1))^\ell.
\]

\begin{thm}\label{induced}
Assume the conditions of Theorem~\ref{bigtheorem} and in addition
that $m=O(n^{1/2+\eps})$ and $x_{m+1}=\cdots=x_n=0$.
Then the probability that a random graph with degree
sequence~$\dvec$ has $\XX^{[m]}$ as an induced subgraph is
\begin{align*}
  \lambda^X & (1-\lambda)^{\binom{m}{2}-X} \\
  &{}\times\exp\biggl(
      \frac{2\omega_{1,1}-\omega_{0,2}}{4An}
       + \frac{m^2}{2n} + \frac{(1{-}2\lambda)\omega_{0,1}}{4An}
       + \frac{4\omega_{1,0}\omega_{0,1}-\omega_{0,1}^2
                        -2\omega_{1,0}^2}{8An^2}\\
     &{\kern 15mm} 
      + \frac{(2\omega_{1,1}-\omega_{2,0}-\omega_{0,2})m}{4An^2}
      - \frac{(1{-}2\lambda)(\omega_{0,3}+3\omega_{2,1}-3\omega_{1,2})}
      	     {24A^2n^2}
	  + O(n^{-b})\biggr).
\end{align*}
\end{thm}
\noindent Note that, within the stated error term, the probability is
independent of $d_{m+1},\ldots,d_n$ except inasmuch as they 
contribute to $d$ and~$\lambda$.

The factor $\lambda^X(1-\lambda)^{\binom{m}{2}-X}$ in Theorem~\ref{induced}
is the probability for an ordinary Erd\H os-R\'enyi
random graph with edge probability~$\lambda$.
A sufficient condition for the argument of the exponential to be $o(1)$ is
$m^2(m+\max_{j=1}^m\abs{d_j-d})=o(An)$.
Relaxing this condition by a factor of $m$ allows us to see the leading terms of
the deviation of behaviour from an ordinary random graph. 

\begin{cor}\label{inducedcor}
Assume the conditions of Theorem~\ref{bigtheorem} and also
that $m(m+\max_{j=1}^m\abs{d_j-d})=o(An)$ and $x_{m+1}=\cdots=x_n=0$.
Then the probability that a random graph with degree
sequence~$\dvec$ has $\XX^{[m]}$ as an induced subgraph is
\[
 \lambda^X (1-\lambda)^{\binom{m}{2}-X}
  \exp\biggl( \frac{\omega_{1,1}}{2An} - \frac{\omega_{0,2}}{4An} + o(1)\biggr).
\]
\end{cor}
The term $\omega_{0,2}/(4An)$ was obtained in~\cite{KSW09} in the case that
$\dvec=((n-1)/2,\ldots,(n-1)/2)$, $m=o(\sqrt n)$, provided
$x_1,\ldots,x_m$ don't differ too
much from $\lambda m$. (In the regular case $\omega_{1,1}=0$.)

It has been shown by Barvinok and Hartigan~\cite{BH}, see also~\cite{CPA}, that
an independent-edge model more accurately matching the $\dvec$ model has
each edge $jk$ chosen with probability~$\lambda_{jk}$, where these constants
were introduced in~\cite{MWreg} and will appear generalised in the following section.
We will call this the \textit{$\{\lambda_{jk}\}$-model}.
Under our strict constraints on $\dvec$, we have
\begin{equation}\label{lamnull}
  \lambda_{jk} = \lambda + \frac{d_j-d}{n} + \frac{d_k-d}{n}
      + \frac{(1-2\lambda)(d_j-d)(d_k-d)}{2An^2} + \OO(n^{-3/2}).
\end{equation}
We can restate Theorem~\ref{induced} with that model in mind.

\begin{cor}\label{induced2}
Assume the conditions of Theorem~\ref{bigtheorem} and in addition
that $m=O(n^{1/2+\eps})$ and $x_{m+1}=\cdots=x_n=0$.
Then the probability that a random graph with degree
sequence~$\dvec$ has $\XX^{[m]}$ as an induced subgraph is
\begin{align*}
  \prod_{jk\in\XX} \lambda_{jk} &
  \prod_{jk\notin\XX} (1-\lambda_{jk})\\
  &{}\times\exp\biggl(
      - \frac{\omega_{0,2}}{4An}
       + \frac{m^2}{2n} + \frac{(1{-}2\lambda)\omega_{0,1}}{4An}
       + \frac{4\omega_{1,0}\omega_{0,1}-\omega_{0,1}^2
                        }{8An^2}\\
     &{\kern 15mm} 
      + \frac{(2\omega_{1,1}-\omega_{0,2})m}{4An^2}
      - \frac{(1{-}2\lambda)(\omega_{0,3}-3\omega_{1,2})}
      	     {24A^2n^2}
	  + O(n^{-b})\biggr),
\end{align*}
where the two products are restricted to $1\le j<k\le m$.
\end{cor}
Note that the $\omega_{1,1}$ term of Corollary~\ref{inducedcor}
has disappeared, but the $\omega_{0,2}$ term remains. The
exponential factor quantifies how much the $\{\lambda_{jk}\}$-model
is in error. However, when $\XX$ is generated according to the
$\{\lambda_{jk}\}$-model, the expectation of the argument of the
exponential is $O(n^{-b})$ and the variance is $\OO(n^{-1/2})$,
so there is some sense in which we can say that the
$\{\lambda_{jk}\}$-model gives a very accurate estimate
for \textit{typical\/} subgraphs. The details of this remain to be
worked out.

Let $\YY$ be any graph on $n$ vertices and let $Y$ be its number of edges.
Define the random variable $E_{\dvec,\YY}$ to be the number of edges that
a random graph with degree sequence $\dvec$ has in common with~$\YY$.
Barvinok and Hartigan~\cite{BH} proved that $E_{\dvec,\YY}$ is concentrated
close to $\sum_{jk\in\YY} \lambda_{jk}$ when $Y=\Omega(n^2)$.  In fact their
estimate is explicit enough to infer this concentration with weaker bounds
for $Y=\omega(n^{3/2}\log n)$.
Under our strict conditions on $\dvec$, we can obtain such a result for all~$Y$.
We have not determined the best result that follows from Theorem~\ref{bigtheorem},
but will for this paper be content with the following weak corollary of
Theorem~\ref{induced}.

\begin{cor}\label{distrib}
Assume the conditions of Theorem~\ref{bigtheorem}.  If\/ $Y=O(n^{1/5})$, then
\begin{equation}\label{Zbinom}
 \Prob(E_{\dvec,\YY})=k) 
   = \binom{Y}{k}\lambda^k(1-\lambda)^{\binom{Y}{2}-k}
      \(1 + \OO(n^{-1/5})\)
\end{equation}
uniformly over~$k=0,\ldots,Y$.
If\/ $Y=\Omega(n^{1/5})$, then
\[
    E_{\dvec,\YY} = \lambda Y(1+O(n^{-1/10+\delta}))
\] with probability $1 - O(\exp(n^{-\delta}))$ for any $\delta\in(0,\tfrac{1}{10})$.
\end{cor}

\begin{proof}
If $Y=O(n^{1/5})$, there are $O(n^{1/5})$ vertices incident with 
edges of $\YY$, so Theorem~\ref{induced} shows that $Z$ has a binomial
distribution with the precision given by~\eqref{Zbinom}.
If $Y=\Omega(n^{1/5})$,
divide the edges of $\YY$ into subsets of size $\Theta(n^{1/5})$ and
apply~\eqref{Zbinom} and Chernoff's Inequality to bound the number of
edges in each subset that are in common with the random graph.
(Clearly the constants in this corollary can be tuned in various ways.)
\end{proof}

\medskip
The theorems above should be enough to allow
transfer of quite a lot of the theory of ordinary random graphs to 
dense random graphs with given degrees.
However, our purpose in this paper is to develop the tools rather than
to explore the applications in detail.
We will be content with some simple illustrations.

\begin{thm}\label{sptrees}
Let $\dvec=(d,d,\ldots,d)$ satisfy the conditions of
Theorem~\ref{bigtheorem}.  Then for a random $d$-regular graph,
we have the following.
(Note that in each case the quantity in front of the exponential
is the expectation for ordinary random graphs with edge
probability~$\lambda$.)
\begin{enumerate}
\itemsep=0pt
 \item[(a)]  If $n$ is even, the expected number of perfect matchings is 
   \[
      \frac{\lambda^{n/2} n!}{2^{n/2}(n/2)!} 
        \exp\biggl( \frac{1{-}\lambda}{4\lambda}+O(n^{-b})\biggr).
   \]
 \item[(b)]  If $q=q(n)$ is a integer function such that 
 $3\le q\le n$, then the expected number of $q$-cycles is
   \[
      \frac{\lambda^{q} n!}{2q(n{-}q)!} \exp\biggl(
         - \frac{(1{-}\lambda)q(n{-}q)}{\lambda n^2} + O(n^{-b})\biggr).
   \]
 \item[(c)] The expected number of spanning trees is
   \[
     n^{n-2} \lambda^{n-1}
     \exp\biggl( \frac{7(1{-}\lambda)}{2\lambda}
        + O(n^{-b})\biggr).
   \]
\end{enumerate}
\end{thm}

\begin{proof}
Parts (a) and (b) follow immediately from Corollaries~\ref{reg}
and~\ref{flat}, respectively.

Part~(c) is not so simple since trees have various degree sequences
and those with maximum degree greater than $n^{1/2+\eps}$ do not
satisfy the requirements of Theorem~\ref{bigtheorem}.
Let $\T$ be the set of all labelled trees with $n$ vertices.
If $\xvec$ denotes the degree sequence of a member of $\T$,
let $\T_1$ be the subset of $\T$ with
$\xmax\le n^\eps$ and let $\T_2=\T\setminus \T_1$.

For sufficiently small $\eps>0$, the following are true.
\begin{enumerate}
\itemsep=0pt
\item[(i)] The probability in $\T$
  that the maximum degree exceeds $k$ is
  at most $2n/k!$ for any integer $k\ge 0$.
\item[(ii)] $X=n-1$.
\item[(iii)] In $\T_1$ we have that $X_2=5n+O(n^{1/2+3\eps})$ with
  probability $1-O(e^{-n^\eps})$.
\item[(iv)] In $\T_1$ we have $X_3=\OO(n)$ and $H = \OO(n)$.
\item[(v)] 
  The sum of $\lambda^{-\sum_{j=1}^n\max\{0,x_j-n^\eps\,\}}$ over  
  $\T_2$ is $\OO(1)\, n^{n-1}/(n^\eps)!\,$.
\end{enumerate}

Facts (i) and (iv) follow from the
well-known generating function for labelled trees by degree
sequence, which is
\[
    z_1z_2\cdots z_n(z_1+z_2+\cdots+z_n)^{n-2}.
\]

To obtain (iii), note that the same 
probabilities occur if we take $x_j=1+Y_j$ for $1\le j\le n$,
where $Y_1,\ldots,Y_n$ are independent Poisson variates
with mean~1 truncated at $n^\eps$, subject to having sum $n-2$.
(This is a standard
property of multinomial distributions; any mean will do.)
Now we can write
\[
   \Prob\Bigl(\bigl| {\textstyle\sum} Y_j^2
                 - 2n\bigr| \ge n^{1/2+3\eps} \Bigm|
         {\textstyle\sum} Y_j = n-2\Bigr)
    \le 
    \frac{\Prob\Bigl(\bigl|\sum Y_j^2-2n\bigr|
           \ge n^{1/2+3\eps}\Bigr)}
         {\Prob \Bigl(\,\sum Y_j = n-2\Bigr)}.
\]
The expectation of $Y_j^2$ is $2+O(n^{-1})$. 
Now bound the numerator by applying a concentration inequality
like Hoefffing's~\cite{hoeffding} and the denominator
by noting that $\sum Y_j$ would be a
Poisson distribution with mean~$n$ except for the truncation.
Item (iii) follows.

To obtain (v), note that 
$\sum_{j=1}^n\max\{0,x_j-n^\eps\,\} \ge \Deltait-n^\eps$
for trees with maximum degree $\Deltait$, and the number of such trees is
bounded by (i) with $k=\Delta-1$.  Summing over $\Delta>n^\eps$ gives
the desired bound.

Now we bound the expected number of trees in $\T_2$ that
appear in a random $d$-regular graph. 
For a tree $T\in\T_2$, let $F(T)$ be any forest
obtained by deleting all but $\lceil n^\eps\rceil$ edges from each vertex
that has degree greater than $n^\eps$.
Then the probability that $T$ appears is bounded by the probability
that $F(T)$ appears.  Moreover, 
$F(T)$ satisfies the requirements of Corollary~\ref{flat} and has
at least $n-1-\sum_{j=1}^n\max\{0,x_j-n^\eps\}$ edges.
Corollary~\ref{flat} gives $\hit(\dvec,F(T))=e^{O(n^\eps)}$.
Applying fact (v), we find that the
expected number of these trees 
easily falls within the error term of part (c) of the theorem.

Finally, the trees in $\T_1$ all satisfy the conditions of
Corollary~\ref{flat} and have $\hit(\dvec,\XX) = O(n^3)$.
Those with $X_2=5n+O(n^{1/2+3\eps})$ have
$\hit(\dvec,\XX)=\exp\(7(1{-}\lambda)/(2\lambda)+O(n^{-b})\)$.
Part (c) of the theorem now follows from items (i) and (iii).
\end{proof}

The average number of spanning trees in random regular graphs of
bounded degree was studied in~\cite{Barbados}. 

\nicebreak
\section{Proof of Theorem~\ref{bigtheorem}}\label{s:calculations}

In this section we express $G(\dvec,\XX)$ as a contour integral in
$n$-dimensional complex space, then estimate its value
using the saddle-point method.

We will use a shorthand notation for summation
over doubly subscripted variables.  From the matrix
$\XX=(x_{jk})$, define sets
\[
   \XX(j) = \{\, k \suchthat 1\le k\le n,\, x_{jk}=1\,\},\quad
 \barXX(j) = \{\, k \suchthat 1\le k\le n,\, x_{jk}=0,\, k\ne j\,\}
\]
for $1\le j\le n$.  Note that $j\notin \XX(j),\barXX(j)$,
also that $\card{\XX(j)}=x_j$ and $\card{\barXX(j)}=n-1-x_j$.
If~$z_{jk}$ is a symmetric variable for
$1\leq j,k\leq n$, we define
\begin{align*}
 && z_{j\c} &= \sum_{k=1}^{n} z_{jk},&
   z_{\c\c} &= \sum_{j=1}^{n} \sum_{k=1}^{n} z_{jk}, && \displaybreak[0]\\
   && z_{j\c|\XX} &= \sum_{k\in \XX(j)} z_{jk},&
   z_{\c\c|\XX} &= \sum_{jk\in\XX} z_{jk}, && \\
  && z_{j\c|\barXX} &= \sum_{k\in \barXX(j)} z_{jk},&
   z_{\c\c|\barXX} &= \sum_{jk\in\barXX} z_{jk}. &&
\end{align*}
There is some slight lack of symmetry in the definitions.
To clarify, we note that
\[
 \sum_j z_{j\c|\XX} = 2z_{\c\c|\XX}\text{~~and~~}
 \sum_j z_{j\c|\barXX} = 2z_{\c\c|\barXX},\text{~~but~~}
 \sum_j z_{j\c} = z_{\c\c}.
\]

Firstly, notice that $G(\dvec,\XX)$ is the coefficient of
$z_1^{d_1}z_2^{d_2}\cdots z_n^{d_n}$ in the function
\[ \prod_{jk\in\barXX} \,(1+z_jz_k).\]
By Cauchy's theorem this equals
\[ G(\dvec,\XX) = \frac{1}{(2\pi i)^n} \oint \cdots \oint
  \frac{\prod_{jk\in\barXX}\,
              (1+z_jz_k)}{z_1^{d_1+1}\cdots z_n^{d_n+1}} \,
              dz_1 \cdots dz_n,
\]
where each integral is along a simple closed contour enclosing the origin
anticlockwise.
It~will suffice to take each contour to be a circle;  specifically,
we will write
\[ z_j = r_j e^{i\theta_j}\]
for $1\leq j\leq n$.  Also define
\[ \lambda_{jk} = \frac{r_jr_k}{1+r_jr_k} \]
for $1\le j,k\le n$.   Then
\begin{equation}\label{start}
 G(\dvec,\XX) = \frac{\prod_{jk\in\barXX}\, (1 + r_jr_k)}
{(2\pi)^n\prod_{j=1}^n r_j^{d_j}}
  \int_{-\pi}^{\pi}\!\!\cdots \int_{-\pi}^\pi \frac{\prod_{jk\in\barXX}
    \(1 + \lambda_{jk}(e^{i(\theta_j+\theta_k)}-1)\)}
  {\exp(i\sum_{j=1}^n d_j\theta_j)}
   \, d\thetavec,
\end{equation}
where $\thetavec=(\theta_1,\ldots,\theta_n)$.
Write $G(\dvec,\XX) = P(\dvec,\XX) I(\dvec,\XX)$
where $P(\dvec,\XX)$
denotes the factor in front of the integral in~\eqref{start}
and $I(\dvec,\XX)$ denotes the integral.
We will choose the radii
$r_j$ so that there is no linear term in the logarithm
of the integrand of $I(\dvec,\XX)$ when expanded for
small $\thetavec$.  The linear term is
\[
 \sum_{jk\in\barXX}\lambda_{jk} (\theta_j+\theta_k)
 - \sum_{j=1}^n d_j\theta_j.
\]
For this to vanish for all $\thetavec$, we require
\begin{equation}\label{rad1}
  \lambda_{j\c|\barXX} = d_j \quad (1\le j\le n).
\end{equation}
Although it is not hard to show that \eqref{rad1} has an exact
solution, we can get by with a near-solution since \eqref{start}
is valid for all positive radii.
In Section~\ref{s:radii} we find such a near-solution
and determine to sufficient accuracy the various functions of
the radii, such as $P(\dvec,\XX)$, that we require.
In Section~\ref{s:integral} we evaluate the integral
$I(\dvec,\XX)$ within a certain region~$\R$ 
defined in~\eqref{Rdef}.
Section~\ref{s:boxing} notes that the contribution to
the integral from the region outside of $\R$ and its
translate $\R+(\pi,\ldots,\pi)$ is minor in comparison.

\nicebreak
\subsection{Locating the saddle-point}\label{s:radii}

In this section we derive a near-solution of \eqref{rad1}
and record some of the consequences.  As with the whole paper,
we work under the assumptions of Theorem~\ref{bigtheorem}.

Change variables to $\{a_j\}_{j=1}^n$ as follows:
\begin{equation}\label{qrdef}
  r_j = r\frac{1+a_j}{1-r^2a_j},
\end{equation}
where
\[
  r= \sqrt{\frac{\lambda}{1-\lambda}}\;.
\]
{}From \eqref{qrdef} we find that
\begin{equation}\label{rad2}
  \lambda_{jk}/\lambda = 1 + a_j + a_k + Z_{jk},
\end{equation}
where
\begin{equation}\label{Zjk}
 Z_{jk} = \frac{a_ja_k(1-r^2-r^2a_j - r^2a_k)}
               {1+r^2a_ja_k},
\end{equation}
and that equation~\eqref{rad1} can be rewritten as
\begin{equation}\label{rad3}
\frac{\delta_j}{\lambda} = 
  (n-1)a_j - a_jx_j + \sum_{k\in\barXX(j)} a_k + Z_{j\c|\barXX}.
\end{equation}

\noindent Summing \eqref{rad3} over all $j$, we find that
\begin{equation}\label{rad5}
 X = \sum_{j=1}^n \((n-1)a_j - a_jx_j\)
      + Z_{\c\c|\barXX}\,.
\end{equation}
Replace the term $\sum_{k\in\barXX(j)} a_k$ in \eqref{rad3}
by $\sum_{k=1}^n a_k - \sum_{k\in\XX(j)} a_k - a_j$,
and substitute the value
\[
  \sum_{k=1}^n a_k = \frac{1}{n}\sum_{k=1}^n \(a_k+a_kx_k\)
	+ \frac{X}{n} - \frac{1}{n}Z_{\c\c|\barXX}
\]
implied by~\eqref{rad5}.
After some rearrangement, we find that
$a_j = \mathbb{A}_j(a_1,\ldots,a_n)$ for each~$j$, where
\begin{equation}\label{rad6}
\begin{split}
   \mathbb{A}_j(a_1,\ldots,a_n) 
     &= \frac{\delta_j}{\lambda n} + \frac{2a_j + a_jx_j}{n}
       - \frac{X}{n^2}  
       - \frac{1}{n^2}\sum_{k=1}^n \(a_k+a_kx_k)\\
      &{\quad} + \frac{1}{n}\sum_{k\in \XX(j)} a_k
       - \frac{1}{n} Z_{j\c|\barXX}
       + \frac{1}{n^2} Z_{\c\c|\barXX}\,.
\end{split}
\end{equation}
In the vicinity of $\avec=(0,0,\ldots,0)$, the iteration
$\avec := \(\mathbb{A}_1(\avec),\ldots,\mathbb{A}_n(\avec)\)$
is a contraction mapping that converges to a solution of
\eqref{rad1}, as can be proved using the method demonstrated
in~\cite{CGM}.
However, as noted above, we do not need to solve~\eqref{rad1} exactly
but will work with an approximate solution.
Hopefully without confusing the reader, from now on we will use
$\avec$ to denote the result of four iterations starting at $\avec=(0,0,\ldots,0)$.
We will also write $Z_{jk}$ and $\lambda_{jk}$ to mean the values
implied by~\eqref{rad2} and~\eqref{Zjk} for our chosen~$\avec$.
Applying~\eqref{rad6} four times, we find
\begin{equation}\label{rad7}
  a_j = \frac{\delta_j}{\lambda n}
  		 + \frac{\delta_jx_j}{\lambda n^2}
		 - \frac{X}{n^2}
		 + \frac{B_j}{\lambda n^2}
		 + \cdots+ \OO(n^{-5/2}),
\end{equation}
where the ellipsis conceals about 60 terms of order
$\OO(n^{-3/2})$.  Most of the terms involve counts of
subgraphs of $\XX$ up to order~5, with the vertices
weighted by powers of the numbers $\{\delta_j\}$.
This implies an expansion
\[
   Z_{jk} = \frac{\delta_j\delta_k(1-2\lambda)}{2\lambda An^2}
            + \cdots + \OO(n^{-5/2}).
\]

The value of $\lambda_{jk}$ is given by substituting
estimate~\eqref{rad7} into~\eqref{rad2}.  In particular,
uniformly over $j$,
\begin{equation}\label{rad8}
    \lambda_{j\c|\barXX} = d_j + \OO(n^{-3/2}).
\end{equation}
Define $\alpha_{jk},\beta_{jk},\gamma_{jk}$ by
$\alpha_{jk}=\beta_{jk}=\gamma_{jk}=0$ if $j=k$ and
\begin{align}\label{ABGdef}
\begin{split}
\dfrac{1}{2}\lambda_{jk}(1-\lambda_{jk}) &=  A + \alpha_{jk}, \\
\dfrac{1}{6}\lambda_{jk}(1-\lambda_{jk})(1-2\lambda_{jk})
   &= A_3 + \beta_{jk}, \\
\dfrac{1}{24}\lambda_{jk}(1-\lambda_{jk})(1-6\lambda_{jk}+6\lambda_{jk}^2)
   &= A_4 + \gamma_{jk},
\end{split}
\end{align}
if $j\ne k$, where
\[ A = \dfrac{1}{2}\lambda(1-\lambda),~
 A_3 = \dfrac{1}{6}\lambda(1-\lambda)(1-2\lambda), \text{~and~}
 A_4 = \dfrac{1}{24}\lambda(1-\lambda)(1-6\lambda + 6\lambda^2).
\]

In~evaluating
the integral $I(\dvec,\XX)$, the following approximations
of $\alpha_{jk}$, $\beta_{jk}$, and $\gamma_{jk}$
for $j\ne k$ will be required:
\begin{gather}
\begin{split}\label{rad10}
  \alpha_{jk} &= \frac{(1-2\lambda)(\delta_j+\delta_k)}{2n}
       - \frac{\delta_j^2+\delta_k^2}{2n^2}
       + \frac{(1-12A)\delta_j\delta_k}{4An^2}
       + \frac{(1-2\lambda)(B_j+B_k)}{2n^2}\\
    &{\quad} - \frac{\lambda(1-2\lambda)X}{n^2}
     + \frac{(1-2\lambda)(\delta_j x_j+\delta_k x_k)}{2n^2}
     + \OO(n^{-3/2}),
\end{split} \displaybreak[0]\\[1.5ex]
\begin{split}\label{rad11}
  \beta_{jk}
 &= \frac{(1-12A)(\delta_j+\delta_k)}{6n}
        + \OO(n^{-1}),
\end{split} \\[1.5ex]
\begin{split}\label{rad12}
  \gamma_{jk}
     &= \OO(n^{-1/2}).
\end{split}
\end{gather}
We will also need the following summations.
\begin{align}
   \alpha_{j\c} &= \dfrac12 (1-2\lambda)\delta_j
                  - \frac{\delta_j^2}{2n} - \frac{R_2}{2n^2}
                  + \frac{(1-2\lambda)(\delta_j x_j + B_j)}{2n} 
                   + \OO(n^{-1/2})\label{asum1}\\
   \alpha_{\c\c} &= -\frac{R_2}{n} + \lambda(1-2\lambda)X
		   + \OO(n^{1/2})\label{asum2}\\
   \beta_{j\c} &= \dfrac16 (1-12A)\delta_j
                   + \OO(1)\label{asum4}\\
   \beta_{\c\c} &= \OO(n)\label{asum5}
\end{align}

\nicebreak
\subsection{Estimating the factor $P(\dvec,\XX)$}\label{s:front}

Let
\[\Lambda = \prod_{jk\in\barXX}
         \lambda_{jk}^{\lambda_{jk}} (1-\lambda_{jk})^{1-\lambda_{jk}}.\]
Then
\begin{align*}
\Lambda^{-1} &= \prod_{jk\in\barXX}
  \biggl(\biggl(\frac{1 + r_jr_k}{r_j r_k}\biggr)^{\!\!\lambda_{jk}}
        (1 + r_j r_k)^{1-\lambda_{jk}}\biggr) \\
 &= \prod_{jk\in\barXX} (1 + r_j r_k)\,
   \prod_{j=1}^n r_j^{-\lambda_{j\c|\barXX}}\\
 &= \prod_{jk\in\barXX} (1+r_j r_k)\,
  \prod_{j=1}^n r_j^{-d_j+\OO(n^{-3/2})}
\end{align*}
using \eqref{rad8}.  Therefore the factor
$P(\dvec,\XX)$ in front
of the integral in \eqref{start} is given by
\begin{equation}\label{Pdx}
  P(\dvec,\XX) = (2\pi)^{-n} \,\Lambda^{-1}\exp\(\OO(n^{-1/2})\).
\end{equation}
We proceed to estimate $\Lambda$.
Writing $\lambda_{jk}=\lambda(1+z_{jk})$, we have
\begin{equation}\label{rad13}
\begin{split}
  &\kern-2mm\log\biggl( \frac{\lambda_{jk}^{\lambda_{jk}}
     (1-\lambda_{jk})^{1-\lambda_{jk}}}
                 {\lambda^\lambda(1-\lambda)^{1-\lambda}}\biggr)
     = \lambda z_{jk}\log\biggl(\frac{\lambda}{1-\lambda}\biggr)\\
      &{\qquad} +
     \frac{\lambda}{2(1-\lambda)}z_{jk}^2 - \frac{\lambda(1-2\lambda)}
          {6(1-\lambda)^2} z_{jk}^3
     + \frac{\lambda(1-3\lambda+3\lambda^2)}{12(1-\lambda)^3} z_{jk}^4
     + \OO(n^{-5/2}).
\end{split}
\end{equation}
We know from \eqref{rad8} that
$\lambda_{\c\c|\barXX} = E + \OO(n^{-1/2})$, which implies that
$z_{\c\c|\barXX}=X+\OO(n^{-1/2})$,
hence the first term on the right side of~\eqref{rad13} contributes
$\lambda^{\lambda X}(1-\lambda)^{-\lambda X}\exp\(\OO(n^{-1/2})\)$
to~$\Lambda$.
Now using \eqref{rad2},
and recalling that $\card\barXX = \binom{n}{2}-X$,
we can write $z_{jk} = a_j + a_k + Z_{jk}$ and apply the
estimates in the previous subsection to obtain
\begin{align}\label{rad14}
\begin{split}
  \Lambda
  &= \( \lambda^\lambda(1-\lambda)^{1-\lambda} \)^{\binom{n}{2}}
   (1-\lambda)^{-X}\\
  &{\qquad}\times 
   \exp\biggr( \frac{(n+2)R_2}{4An^2} 
       - \frac{(1-2\lambda)R_3}{24A^2n^2} 
       + \frac{(1-6A)R_4}{96A^3n^3}\\
  &{\kern2.3cm} 
           + \frac{C_{2,1}+2D}{4An^2}
           - \frac{\lambda^2 X^2}{2An^2}
                + \frac{R_2^2}{16A^2n^4}
                + \OO(n^{-1/2})
       \biggl).
\end{split}
\end{align}

As in \cite{MWreg}, our answer will be simpler when written
in terms of binomial coefficients.  Using Stirling's formula
or otherwise we find that
\begin{align}\label{binom}
\begin{split}
  \prod_{j=1}^n \binom{n{-}x_j{-}1}{d_j}
    &= (2\pi n)^{-n/2}
      \lambda^{-n/2-\lambda n(n-1)}
      (1-\lambda)^{-n/2-(1-\lambda)n(n-1)+2X}\\[-1ex]
  &{\qquad}\times 
   \exp\biggr( -\frac{1-14A}{24A} -\frac{R_2}{4An}
       + \frac{(1-2\lambda)^2R_2}{16A^2n^2}
       + \frac{(1-2\lambda)R_3}{24A^2n^2} \\
  &{\kern25mm} - \frac{(1-6A)R_4}{96A^3n^3}
           - \frac{C_{2,1}}{4An^2}
           + \frac{\lambda X}{(1-\lambda)n}
                + \OO(n^{-1/2})
       \biggl).
     \end{split}
\end{align}

Combining \eqref{Pdx}, \eqref{rad14} and \eqref{binom}, we find that
\begin{align}\label{Pvalue}
\begin{split}
   P(\dvec,\XX) &= \prod_{j=1}^n \binom{n{-}x_j{-}1}{d_j}\;
    \biggl( \frac{\pi}{An}\biggr)^{\!\!-n/2}
    \(\lambda^\lambda(1-\lambda)^{1-\lambda}\)^{\binom{n}{2}} \\
  &{\kern10mm}\times\exp\biggl(
         \frac{1-14A}{24A} - \frac{R_2}{16A^2n^2}
         - \frac{R_2^2}{16A^2n^4} \\
	&{\kern30mm} - \frac{\lambda X}{(1{-}\lambda)n}
        + \frac{\lambda X^2}{(1{-}\lambda)n^2}
        - \frac{D}{2An^2}
	+ O(n^{-b})
    \biggr).
\end{split}
\end{align}

\nicebreak
\subsection{Estimating the main part of the integral}\label{s:integral}

Our next task is to evaluate the main part of the
integral $I(\dvec,\XX)$ given by
\begin{equation}
I(\dvec,\XX) = \int_{-\pi}^{\pi}\!\cdots \int_{-\pi}^\pi \frac{\prod_{jk\in\barXX}
    \(1 + \lambda_{jk}(e^{i(\theta_j+\theta_k)}-1)\)}
  {\exp(i\sum_{j=1}^n d_j\theta_j)}
   \, d\thetavec\,.
\label{FDef}
\end{equation}

It will be established in this section and the next that 
the value of the integral is concentrated near the places
where the integrand has the largest absolute value.  This
happens at the two points $\thetavec=(0,0,\ldots,0)$ and
$\thetavec=(\pi,\pi,\ldots,\pi)$. These two points are
equivalent, since the integrand is unchanged under the
mapping $\thetavec\mapsto\thetavec+(\pi,\pi,\ldots,\pi)$.
(This requires the fact that $\sum d_j$ is even; otherwise
the mapping changes the sign of the integrand and the integral
is zero as it should~be.)
Consequently, in this section we will focus on a neighbourhood
of $(0,0,\ldots,0)$, specifically the hypercube $\R$ defined by
\begin{equation}\label{Rdef}
   \R = \bigl\{\, \thetavec \Suchthat
      \abs{\theta_j}\le n^{-1/2+\eps},
            1\le j\le n \,\bigr\}.
\end{equation}

Let $F(\thetavec)$ be the integrand of \eqref{FDef}.
We are going to establish the following.
\begin{lemma}\label{Fintegral}
Under the conditions of Theorem~\ref{bigtheorem},  there is a
region $\S$ with $\R\subseteq\S\subseteq 4\R$ such that
\begin{equation}\label{Fint}
\int_{\S} F(\thetavec)\,d\thetavec = 2^{-1/2}\,
 \biggl(\frac{\pi}{An}\biggr)^{\!\!n/2}
  \exp\biggl( -\frac{1-20A}{24A} + \frac{\lambda X}{(1-\lambda)n}
   + \frac{R_2}{16A^2n^2} + O(n^{-b}) \biggr).
\end{equation}
\end{lemma}

\medskip

In a region $O(1)\R$, we can expand
\begin{align*}
 F(\thetavec) &= \exp\biggl( -\sum_{jk\in\barXX}
        (A+\alpha_{jk})(\theta_j+\theta_k)^2
       - i\sum_{jk\in\barXX}(A_3+\beta_{jk})(\theta_j+\theta_k)^3 \\
     & {\kern 4em}
      + \sum_{jk\in\barXX}(A_4+\gamma_{jk})(\theta_j+\theta_k)^4
      + \OO\Bigl(n^{-1/2} +
          A\,\sum_{jk\in\barXX}\,\abs{\theta_j+\theta_k}^5\,\Bigr)
              \biggr). \\
  &= \exp\biggl( -\sumjk
        (A+\alpha_{jk})(\theta_j+\theta_k)^2
                  - i\sumjk(A_3+\beta_{jk})(\theta_j+\theta_k)^3 \\
     & {\kern 4em}
                  + \sumjk A_4(\theta_j+\theta_k)^4
                  + \sum_{jk\in\XX} A\,(\theta_j+\theta_k)^2
                  + \OO(n^{-1/2})
              \biggr),
\end{align*}
where $A$, $A_3$, $A_4$, $\alpha_{jk}$, $\beta_{jk}$,
and $\gamma_{jk}$ were defined in~\eqref{ABGdef}.
Approximations for $\alpha_{jk}$, $\beta_{jk}$, $\gamma_{jk}$ were
given in \eqref{rad10}--\eqref{rad12}.  Note that
$\alpha_{jk}, \beta_{jk}, \gamma_{jk}=\OO(n^{-1/2})$ uniformly
over~$j,k$.

We will transform the integral to diagonalize the quadratic
terms, proceeding in two steps.  The first step will diagonalize
the quadratic form $\sumjk (\theta_j+\theta_k)^2$, and the second
will complete the diagonalization.

\nicebreak
\subsection{First change of variables}\label{s:change}

We first adopt from \cite{MWreg} a linear transformation that
diagonalizes the quadratic form $\sumjk(\theta_j+\theta_k)^2$.
Define $c$ and $\yvec=(y_1,y_2,\ldots,y_n)$ by
\begin{align}
   c &= 1 - \sqrt{\frac{n-2}{2(n-1)}} = 1-2^{-1/2}+O(n^{-1}) \label{T1a}
       \displaybreak[0] \\
   \theta_j &= y_j - \frac{c}{n}\sum_{k=1}^n y_k\quad(1\le j\le n).\label{T1b}
\end{align}
The transformation $\thetavec=T_1(\yvec)$ defined by~\eqref{T1b}
has determinant $1-c$. 
Also 
\begin{equation}\label{Rbounds}
  (1+c)\R\subseteq T_1^{-1}\R\subseteq(1-c)^{-1}\R.
\end{equation}

For $\ell\ge 1$, define $\mu_\ell=\sum_{j=1}^n y_j^\ell$.
We find the following translations.
\begin{align*}
  \sum_j\theta_j &= (1-c)\mu_1 \\
  \sumjk (\theta_j+\theta_k)^2 &= (n-2)\mu_2 \\
  \smash\sumjk (\theta_j+\theta_k)^3 &= 
     (n-4)\mu_3 + \( 3(1-2c) + 12c/n\)\mu_1\mu_2 \\
        &\quad{}+
          \((-6c+12c^2-4c^3)/n - 4c^2(3-c)/n^2\)\mu_1^3 \\[0.6ex]
  \smash\sumjk (\theta_j+\theta_k)^4 &= 
    (n-8)\mu_4 + 3\mu_2^2 + \(4(1-2c)+32c/n\)\mu_1\mu_3 \\
        &\quad{} -\(24c(1-c)/n+48c^2/n^2\)\mu_1^2\mu_2 \\
        &\quad{} +\( 8c^2(1-c)(3-c)/n^2+8c^3(4-c)/n^3\)\mu_1^4
          \displaybreak[0]\\[0.6ex]
   \sumjk \!\!\alpha_{jk}(\theta_j+\theta_k)^2 &=
    \sum_j \( (1-4c/n)\alpha_{j\c}+2c^2\alpha_{\c\c}/n^2\) y_j^2 \\
        &\quad{} + \sumppd_{j,k}
          \(\alpha_{jk}-4c\alpha_{j\c}/n+2c^2\alpha_{\c\c}/n^2\)y_jy_k
          \displaybreak[0]\\
   \sumjk \!\!\beta_{jk}(\theta_j+\theta_k)^3 &=
     \sum_j\(1-6c/n+12c^2/n^2)\beta_{j\c}-4c^3\beta_{\c\c}/n^3\)y_j^3\\
      &\quad{} +\smash{\sumppd_{j,k}}
       \((3-12c/n)\beta_{jk}-6c(1-4c/n)\beta_{k\c}/n \\
         &{\kern10em}
         +12c^2\beta_{j\c}/n^2 - 12c^3\beta_{\c\c}/n^3\)y_jy_k^2 \\
      &\quad{} + \sumppd_{j,k,\ell}
        \(-6c\beta_{jk}/n+12c^2\beta_{j\c}/n^2-4c^3\beta_{\c\c}/n^3\)
           y_jy_ky_\ell
        \displaybreak[0]\\
    \sum_{jk\in\XX} \,(\theta_j+\theta_k)^2 &=
      \sum_{jk\in\XX} \,(y_j+y_k)^2 - \frac{4c}{n}\mu_1\sum_j x_jy_j
         + \frac{4c^2}{n^2}\mu_1^2X\,.
\end{align*}
In the above, and following, a summation is over
$1,2,\ldots,n$ for each index unless otherwise specified.
Moreover, a prime on the summation symbol (as $\sum'\,$) means that
only terms where the summation indices have distinct values
are included.
For example,
\[
   \sumppd_{j,k} \text{~~means~~}
   \sum_{\substack{1\le j\le n, 1\le k\le n\\j\ne k}}.
\]

Using the size of the hypercube $\R$ together with the bounds
\eqref{asum1}--\eqref{asum5}, we find that whenever
$\thetavec\in O(1)\R$,
we have $F(\thetavec)=G(\yvec)$, where
\begin{align}\label{Gy}
\begin{split}
  G(\yvec) &= -\sum_j \((n-2)A+\alpha_{j\c}-Ax_j\) y_j^2 \\
  &{\quad}+ \smash{\sumppd_{j,k}}
    \( -\alpha_{jk}+2c\alpha_{j\c}/n+2c\alpha_{k\c}/n
         -2c^2\alpha_{\c\c}/n^2 \\
     &{\kern 5em}  +Ax_{jk}-2Acx_j/n-2Acx_k/n+4Ac^2X/n^2\) y_jy_k \\
  &{\quad} -i\sum_j\(nA_3+\beta_{j\c}\)y_j^3 \\
  &{\quad} -i\,\sumppd_{j,k} \( A_3(3-6c)
       +3\beta_{jk}-6c\beta_{k\c}/n\) y_jy_k^2 \\
  &{\quad} -i\,\sumppd_{j,k,\ell} \(A_3(-6c+12c^2-4c^3)/n
         -6c\beta_{jk}/n+12c^2\beta_{j\c}/n^2\) y_jy_ky_\ell \\
  &{\quad} + nA_4\sum_j y_j^4 + 3A_4 \sumppd_{j,k} y_j^2y_k^2 \\
  &{\quad} + 4A_4(1-2c) \sumppd_{j,k} y_jy_k^3
      - 24A_4c(1-c)/n \sumppd_{j,k,\ell} y_jy_ky_\ell^2 \\
  &{\quad} 
   + 8c^2(1-c)(3-c)A_4/n^2 \sumppd_{j,k,\ell,m} y_jy_ky_\ell y_m
      + \OO(n^{-1/2}).
\end{split}
\end{align}

\nicebreak
\subsection{Completing the diagonalization}\label{s:diagonalize}

We now make a second change of variables, $\yvec=T_2 (\zvec)$,
that diagonalizes the quadratic part of $G(\yvec)$,
where $\zvec=(z_1,\ldots,z_n)$.
We will use the method from~\cite{GMX} that is a slight extension
of \cite[Lemma 3.2]{MWtournament}.

\begin{lemma}\label{diagonalize}
Let $\UU$ and $\YY$ be square matrices of the same order, such that
$\UU^{-1}$ exists and all the eigenvalues of
$\UU^{-1} \YY$ are less than 1 in absolute value.
Then
\[ (\II + \YY\UU^{-1})^{-1/2} \,(\UU+\YY)\,
  (\II+\UU^{-1}\YY)^{-1/2} = \UU, \]
where the fractional powers are defined by the binomial expansion.
\quad\qedsymbol
\end{lemma}

If we also have that both $\UU$ and $\YY$ are symmetric, then
$(\II+\YY\UU^{-1})^{-1/2}$ is the transpose of
$(\II + \UU^{-1}\YY)^{-1/2}$, as proved in~\cite{CGM}.

Let $\VV=(v_{jk})$ be the symmetric matrix such that the
quadratic terms of~\eqref{Gy} are $-\yvec\VV\yvec^T$.
We have for all $j\ne k$ that
\begin{align*}
   v_{jj} &= An + \OO(n^{1/2}), \\
   v_{jk} &= Ax_{jk} + \OO(n^{-1/2}).
\end{align*}
Apply Lemma~\ref{diagonalize} with $\VV=\UU+\YY$ where $\UU$
is the diagonal matrix with the same diagonal entries as~$\VV$.
The matrix $\UU^{-1}\YY$ has $jk$-entry equal to
$n^{-1}x_{jk}+\OO(n^{-3/2})$.
Therefore, since the $\infty$-norm (maximum row
sum of absolute values) of $\UU^{-1}\YY$ is $\OO(n^{-1/2})$,
the eigenvalues of $\UU^{-1}\YY$ are all $\OO(n^{-1/2})$.

Let $T_2$ be the transformation given by $T_2(\yvec)=\zvec$,
where $\zvec^T=(\II+\UU^{-1}\YY)^{-1/2}\yvec^T$.
By~\cite[Lemma 2]{CGM}, the Jacobian of $T_2$ is
$1+\OO(n^{-1/2})$.  Expanding $(\II+\UU^{-1}\YY)^{-1/2}$
we find that for $\yvec\in O(1)\R$,
\begin{equation}\label{T2}
  y_j = z_j + \sum_{k=1}^n \(\OO(n^{-3/2})z_k 
        + \OO(n^{-1})x_{jk}\) z_k,
\end{equation}
for each $j$, where the coefficients are uniform and independent
of $\zvec$.  An expression of identical form writes
$\zvec$ in terms of~$\yvec$.
For $\yvec\in O(1)\R$, we find that $G(\yvec)=H(\zvec)$ where
\begin{align}\label{Hz}
\begin{split}
  H(\zvec) &= -\sum_j \((n-2)A+\alpha_{j\c}-Ax_j\) z_j^2
      - i\sum_j\(nA_3+\OO(n^{1/2})\)z_j^3 \\
  &{\quad} -i\,\sumppd_{j,k} \( 3A_3(1-2c) + \OO(n^{-1/2})\) z_jz_k^2
   - i\,\sumppd_{j,k,\ell} \OO(n^{-1}) z_jz_kz_\ell\\
  &{\quad} + nA_4\sum_j \(1+\OO(n^{-1})\) z_j^4
       + 3A_4 \sumppd_{j,k} \(1+\OO(n^{-1})\) z_j^2z_k^2 \\
  &{\quad} + \sumppd_{j,k} \OO(1) z_jz_k^3
      - \sumppd_{j,k,\ell} \OO(n^{-1}) z_jz_kz_\ell^2 
   + \sumppd_{j,k,\ell,m} \OO(n^{-2}) z_jz_kz_\ell z_m
      + \OO(n^{-1/2}),
\end{split}
\end{align}
with only the final expression of the
form $\OO(\,)$ being a function of~$\zvec$.

Now define $\S=T_1^{-1}(T_2^{-1}(2\R))$.  By~\eqref{Rbounds}
and~\eqref{T2}, $\R\subseteq\S\subseteq 4\R$.  Consequently
the conditions for our approximations are satisfied
and~\eqref{Hz} is valid for $\zvec\in 2\R$.

We can now apply Theorem~\ref{MW3} (see Appendix) to estimate
the integral of $H(\zvec)$ over~$2\R$.  We list the coefficients
required.
\begin{align*}
  \mwA &= A
   & \mwD_{jk\ell} &= \OO(1)\\ 
  \mwn &= n
   & \mwE_j &= A_4+\OO(n^{-1})\\
  \mwJ_j &= 0
   & \mwF_{jk} &= 3A_4+\OO(n^{-1}) \displaybreak[0]\\
  \mwa_j &= (2A-\alpha_{j\c}+Ax_j) n^{-1/2}
   & \mwG_{jk} &= \OO(n^{-1/2}) \displaybreak[0]\\
  \mwB_j &= -i A_3 + \OO(n^{-1/2})
   & \mwH_{jk\ell} &= \OO(n^{-1/2})\\
  \mwC_{jk} &= -3(1-\sqrt2) i A_3 + \OO(n^{-1/2})
   & \mwI_{jk\ell m} &= \OO(n^{-1/2})
\end{align*}
We can take $\Deltait=\tfrac34$ and have $\delta(\zvec)=\OO(n^{-1/2})$.
Applying Theorem~\ref{MW3}, we find that
\begin{equation}\label{Hint}
\int_{2\R} H(\zvec)\,d\zvec = \Bigl(\frac{\pi}{An}\Bigr)^{\!\!n/2}
  \exp\Bigl( -\frac{1-20A}{24A} + \frac{\lambda X}{(1-\lambda)n}
   + \frac{R_2}{16A^2n^2} + \OO(n^{-1/2})\mwZ\Bigr),
\end{equation}
where
\[
   \mwZ = \exp\Bigl( \frac{(1-2\lambda)^2}{6A} \Bigr).
\]
{}From the conditions of Theorem~\ref{bigtheorem}, we have that
$\OO(n^{-1/2})\hat Z=O(n^{-b})$.  Lemma~\ref{Fintegral} now follows
on recalling that the Jacobian determinants of $T_1$ and $T_2$
are $\sqrt2+\OO(n^{-1/2})$ and $1+\OO(n^{-1/2})$, respectively.

\nicebreak
\subsection{Bounding the remainder of the integral}\label{s:boxing}

In the previous section, we estimated the value of the integral
$I(\dvec,\XX)$ restricted to a small region $\S\supseteq\R$.
As mentioned
earlier, the integral over $\S+(\pi,\ldots,\pi)$ is the same.
It remains to bound the integral over the remaining parts of
$[-\pi,\pi]^n$.
Define $\R^c=[-\pi,\pi]^n \setminus \( \R\cup (\R+(\pi,\ldots,\pi)\)$.
By employing the same technique as in~\cite{GMX}, but with
the dissection of the region utilised in~\cite{MWreg}, we can
establish the following.  We will omit the proof since no new
techniques are required, but we note for convenience that
\[
 \abs{F(\thetavec)} = \prod_{jk\in\barXX} f_{jk}(A+\alpha_{jk},\theta_j+\theta_k)
\]
where
\[
 f_{jk}(q,z) = \sqrt{1-4q(1-\cos z)}
   \le \exp\( -qz^2 + \dfrac{1}{12}qz^4\).
\]

\begin{lemma}\label{boxing}
  Under the conditions of Theorem~\ref{bigtheorem},
  \[
   \int_{\R^c} \abs{F(\thetavec)}\,d\thetavec
    = O(n^{-1}) \int_{\S} F(\thetavec)\,d\thetavec.
  \]
\end{lemma}

\nicebreak
\section*{Appendix: The value of an integral}\label{s:appendix}

In this appendix we give the value of a certain
multi-dimensional integral.
A similar integral appeared in~\cite{MWreg} and variations of
it appeared in~\cite{Mtourn,euler,MWtournament,GMW}.

This appendix is notationally independent of the rest of the paper.
Summations without explicit limits are over $1,2,\ldots,N$ for
each of the summation indices. A prime on the summation symbol (as
$\sumpp\,)$ indicates that only terms with distinct values of the
summation indices are included.

\begin{thm}\label{MW3}
Let $\eps', \eps'', \eps''', \bar\eps,\check\eps,\Deltait$
be constants
such that $0<\eps'<\eps''<\eps'''$, $\check\eps>0$,
$\bar\eps\ge 0$, and $0<\Deltait<1$.
The following is true if\/ $\eps'''$
and $\bar\eps$ are sufficiently small.\endgraf
Let $\mwA=\mwA(N)$ be a real-valued function such that
$\mwA(N)=\Omega(N^{-\eps'})$.
For $1\le j,k,\ell,m$, let $\mwa_j$, $\mwB_j$,
$\mwC_{jk}$, $\mwD_{jk\ell}$, $\mwE_j$,
$\mwF_{jk}$, $\mwG_{jk}$, $\mwH_{jk\ell}$,
$\mwI_{jk\ell m}$, and $\mwJ_j$ be
complex-valued functions of $N$
such that $\mwa_j,\mwB_j,\ldots,\mwJ_j=O(N^{\bar\eps})$
uniformly over $j,k,\ell,m$.
Suppose that
\begin{align*}
f(\zvec) &= \exp\biggl(
   -\mwA N\sum_j z_j^2
   + \sum_j \mwJ_j z_j 
   + N^{1/2}\sum_j \mwa_j z_j^2 
   + N \sum_j \mwB_j z_j^3
   + \sumppd_{j,k} \mwC_{jk}\,z_j z_k^2 \\
&\kern14mm{}
   + N^{-1} \sumppd_{j,k,\ell} \mwD_{jk\ell}\,z_jz_kz_\ell
   + N \sum_j \mwE_j z_j^4
   + \sumppd_{j,k} \mwF_{jk}\,z_j^2 z_k^2
   + N^{1/2}\sumppd_{j,k} \mwG_{jk} z_j z_k^3 \\
&\kern14mm{}
   + N^{-1/2}\sumppd_{j,k,\ell} \mwH_{jk\ell}\,z_jz_kz_\ell^2
   + N^{-3/2}\sumppd_{j,k,\ell,m} \mwI_{jk\ell m}\,z_jz_kz_\ell z_m
   + \delta(\zvec) \biggr)
\end{align*}
is integrable for $\zvec=(z_1,z_2,\ldots,z_N)\in U_N$
and\/ $\delta(N)=\max_{\zvec\in U_N} \abs{\delta(\zvec)} = o(1)$,
where
\[
U_N = \bigl\{ \zvec \subseteq\Reals^N\Suchthat
              \abs{z_j} \le N^{-1/2+\hat\eps}
  \text{ for\/ $1\le j\le N$}\bigr\},
\]
where $\hat\eps=\hat\eps(N)$ satisfies
$\eps''\le2\hat\eps\le\eps'''$.
Then, provided the $O(\,)$ term in the following
converges to zero,
\[
 \int_{U_N}f(\zvec)\,d\zvec
  = \biggl(\frac{\pi}{\mwA N}\biggr)^{\!N/2}
   \!\exp\( \Thetait_1
     +O\(N^{-1/2+\check\eps}+(N^{-\Deltait}+\delta(N)) \mwZ\)
     \),
\]
where
\begin{align*}
  \Thetait_1 &= \frac{1}{2\mwA N^{1/2}}\sum_j\mwa_j
     + \frac{1}{4\mwA^2N}\sum_j\mwa_j^2
     + \frac{15}{16\mwA^3N}\sum_j\mwB_j^2
     + \frac{3}{8\mwA^3N^2}\sumppd_{j,k}\mwB_j\mwC_{jk} \\
  &\quad{} 
  + \frac{1}{16\mwA^3N^3}\sumppd_{j,k,\ell} \mwC_{jk}\mwC_{j\ell}
    + \frac{3}{4\mwA^2N}\sum_j \mwE_j
    + \frac{1}{4\mwA^2N^2}\sumppd_{j,k} \mwF_{jk} 
  \\
  &\quad{} + \frac{4}{\mwA N}\sum_j \mwJ_j^2
    + \frac{3}{4\mwA^2 N}\sum_j \mwB_j \mwJ_j
    + \frac{1}{4\mwA^2N^2}\sumppd_{j,k} \mwC_{j,k}\mwJ_k
  \displaybreak[0]\\
  \mwZ &= \exp\biggl(
       \frac{1}{4\mwA^2N}\sum_j\Im(\mwa_j)^2
     + \frac{15}{16\mwA^3N}\sum_j\Im(\mwB_j)^2 
     + \frac{3}{8\mwA^3N^2}\sumppd_{j,k}\Im(\mwB_j)\Im(\mwC_{jk}) \\
     &\kern15mm{}
     + \frac{1}{16\mwA^3N^3}\sumppd_{j,k,\ell}
             \Im(\mwC_{jk})\Im(\mwC_{j\ell})
     + \frac{1}{4\mwA N}\sum_j \Im(\mwJ_j)^2
              \\
     &\kern15mm{}
        + \frac{3}{4\mwA^2N}\sum_j \Im(\mwB_j)\Im(\mwJ_j)
     + \frac{1}{4\mwA^2N^2}\sumppd_{j,k} \Im(\mwB_{jk})\Im(\mwJ_k)        
   \biggr).
\end{align*}
\end{thm}

\begin{proof}
The method of proof is the same as in~\cite{CGM}, with extra
terms added.
To simplify the process, we did not explicitly compute the
lower order terms which are presented as~$\Thetait_2$ in~\cite{CGM}.
The details will be omitted.
\end{proof}

\nicebreak

\end{document}